\documentclass[11pt]{article}

\usepackage{amssymb}
\usepackage{amsfonts}
\usepackage{eucal}
\usepackage{amsmath}
\newtheorem{theorem}{Theorem}[section]

\newtheorem{definition}[theorem]{Definition}
\newtheorem{corollary}[theorem]{Corollary}

\newtheorem{lemma}[theorem]{Lemma}
\newtheorem{remark}[theorem]{Remark}

\newenvironment{proof}[1][Proof]{\textbf{#1.} }{\ \rule{0.5em}{0.5em}}
\input pictex.sty

 \usepackage{cancel}
 \usepackage{xcolor}
 
 \usepackage{amsmath,color,ulem}
\newcommand{\stkout}[1]{\ifmmode\text{\sout{\ensuremath{#1}}}\else\sout{#1}\fi}

\DeclareMathOperator\Cov{Cov}

\begin{document}

 \title{Maximal inequalities and convergence results on multidimensionally indexed demimartingales  }
 	\author{Milto Hadjikyriakou\footnote{School of Sciences, University of Central Lancashire, Cyprus campus, 12-14 University Avenue, Pyla, 7080 Larnaka, Cyprus (email:mhadjikyriakou@uclan.ac.uk).}~~ and B.L.S. Prakasa Rao \footnote{CR RAO Advanced Institute of Mathematics, Statistics and Computer Science, Hyderabad 500046, India (e-mail: blsprao@gmail.com).}}	
 \maketitle

\begin{abstract}
We obtain  some maximal probability and moment inequalities for multidimensionally indexed demimartingales. Although the class of single-indexed demimartingales has been studied extensively, no significant amount of work has been done for the corresponding multiindexed class of random variables. This work aims to fill in this gap in the literature by extending well-known inequalities and asymptotic results to this more general class of random variables.
\end{abstract}

\medskip


\section{Introduction}
The notion of demimartingales was introduced by Newman and Wright in \cite{NW1982} in an effort to generalize the concept of positive association. This new class of random variables contains the class of associated random variables, since it can be easily proven that the partial sums of mean zero associated random variables form a sequence of demimartingales. This new class of random objects has been studied extensively over the last few decades; see for example \cite{PR2012} for an extensive discussion on this topic. The notion of multidimensionally indexed demimartingales was introduced in \cite{MH2010}. The idea of random vectors indexed by lattice points is not new and its origin goes back to statistical mechanics and ergodic theory. Moreover, models of several phenomena in statistical physics, crystal physics or Euclidean quantum field theories involve multiple sums of multiindexed random variables. Having in mind all these applications and the observation that the partial sums of mean zero multiindexed associated random variables form a sequence of multiindexed demimartingales, it makes sense to further explore the properties and the asymptotic behavior of this general class of random variables.

\medskip

 Throughout the paper all random variables are defined on a probability space $(\Omega, \mathcal{F}, \mathcal{P})$ and the following notation will be used: let $\mathbb{N}^k$ denote the $k$-dimensional positive integer lattice. For $\mathbf{n}, \mathbf{m}, \in \mathbb{N}^k$ with $\mathbf{n}=\left(n_1, \ldots, n_k\right)$ and $\mathbf{m}=\left(m_1, \ldots, m_k\right)$, we write $\mathbf{n} \leq \mathbf{m}$ if $n_i \leq m_i, i=1, \ldots, k$ and $\mathbf{n}<\mathbf{m}$ if $n_i \leq m_i, i=1, \ldots, k$ with at least one strict inequality. We say that $\mathbf{n} \rightarrow \infty$ if $\displaystyle\min _{1 \leq j \leq k} n_j \rightarrow \infty$ and the notation $I(A)$ is used to denote the indicator function of the set $A$.

\medskip

The paper is structured as follows: in Section 2, we provide some basic definitions and important results that will be useful for the rest of the paper. Section 3 is devoted to Doob-type inequalities and related maximal results. In Section 4, a new Chow-type inequality is discussed together with an asymptotic result and a H\'{a}jek-R\'{e}nyi inequality for multiindexed associated random variables. In Section 5, we present some maximal inequalities given in terms of Orlicz functions while, in Section 6, an uprossing inequality is studied. Finally, in Section 7, we present a Whittle-type inequality and provide a strong law.

\section{Preliminaries}
We start by providing the definitions of multiindexed associated random variables and multiindexed demi(sub)martingales.

\begin{definition}
	A collection of multidimensionally indexed random variables $\left\{X_{\mathbf{i}}, \mathbf{i} \leq \mathbf{n}\right\}$ is said to be associated if for any two coordinatewise nondecreasing functions $f$ and $g$
	$$
	\operatorname{Cov}\left(f\left(X_{\mathbf{i}}, \mathbf{i} \leq \mathbf{n}\right), g\left(X_{\mathbf{i}}, \mathbf{i} \leq \mathbf{n}\right)\right) \geq 0,
	$$
	provided that the covariance is defined. An infinite collection is associated if every finite subcollection is associated.
\end{definition}

\noindent Note that the definition given above is exactly the same as the classical definition of association, stated for the case of multidimensionally indexed random variables, since the index of the variables in no way affects the qualitative property of association, i.e., that nondecreasing functions of all (or some) of the variables are nonnegatively correlated. 

\medskip

The class of multidimensionally indexed demimartingales and demisubmartingales was introduced in \cite{MH2010} (see also \cite{CH2011}). 

\begin{definition}
	A collection of multidimensionally indexed random variables $\left\{S_{\mathbf{n}}, \mathbf{n}\in \mathbb{N}^k\right\}$ is called a multiindexed demimartingale if
	$$
	E\left[\left(S_{\mathbf{j}}-S_{\mathbf{i}}\right) f\left(S_{\mathbf{k}}, \mathbf{k} \leq \mathbf{i}\right)\right] \geq 0
	$$
	for all $\mathbf{i}, \mathbf{j} \in \mathbb{N}^k$ with $\mathbf{i} \leq \mathbf{j}$ and for all componentwise nondecreasing function If, in addition $f$ is required to be nonnegative, then the collection $\left\{S_{\mathbf{n}}, \mathbf{n}\in \mathbb{N}^k\right\}$ is said to be a multiindexed demisubmartingale.
\end{definition}

\begin{remark}
	It can easily be proven that the partial sums of mean zero multiindexed associated random variables form a sequence of multiindexed demimartingales. Furthermore, it is obvious that, with the natural choice of $\sigma$-algebras, a multiindexed (sub)martingale forms a multiindexed demi(sub)martingale.
\end{remark}

The next two results can be found in \cite{WH2009} (see Theorem 3.2 and 3.3 respectively). The first one provides a Doob-type inequality for functions of single-indexed demimartingales while the second one, which is a direct consequence of the Doob's inequality, provides upper bounds for  expectations of maxima of functions of a single-indexed demimartingale. The multiindexed analogues of these two results will be proved in the next section. 

\begin{theorem}
	\label{Doob}Let $S_{1}, S_{2}, \ldots$ be a demimartingale, $g$ be a nonnegative convex function on $\mathbb{R}$ with $g(0)=0$ and $g\left(S_{i}\right) \in L^{1}, i \geqslant 1$. Then for any $\varepsilon>0$,
	$$
	\varepsilon P\left\{\max _{1 \leqslant k \leqslant n} g\left(S_{k}\right) \geqslant \varepsilon\right\} \leqslant  \int_{\left\{\max _{1 \leqslant k \leqslant n} g\left(S_{k}\right) \geqslant \varepsilon\right\}} g\left(S_{n}\right) d P .
	$$
\end{theorem}

\begin{theorem}
	\label{Wang}Let $S_{1}, S_{2}, \ldots$ be a demimartingale and $g$ be a nonnegative convex function on $\mathbb{R}$ with $g(0)=0$. Suppose $E\left(g\left(S_{n}\right)\right)^{p}<\infty$ for $p\geq 1$ and all $n \geqslant 1$. Then for every $n \geqslant 1$,
	$$
	E\left(\max _{1 \leqslant k \leqslant n} g\left(S_{k}\right)\right)^{p} \leqslant\left(\frac{p}{p-1}\right)^{p} E\left(g\left(S_{n}\right)\right)^{p} \quad\mbox{for}\quad p>1
	$$
	and
	$$
	E\left(\max _{1 \leqslant k \leqslant n} g\left(S_{k}\right)\right) \leq \frac{e}{e-1}(1+E\left( g\left(S_{n}\right)\log^+g\left(S_{n}\right)\right)).
	$$
\end{theorem}

\noindent The next lemma will be used in the proof of Theorem \ref{theorem1}.

\begin{lemma}
	Let $f(x_1,x_2,\ldots,x_n) = \max\{x_1,x_2,\ldots,x_n\}$. The function $f$ is componentwise convex and its right derivative with respect to its $i$-th component for $i=1,2,\ldots, n$ is a nonnegative constant.
\end{lemma}
\begin{proof}
	Note that $f$ is convex and hence it is also componentwise convex. Without loss of generality, we will calculate its right derivative with respect to its last component. Define $g(t) = f_{+}^{'}(x_1,x_2,\ldots,x_{n-1},t)$ where $f^{'}_{+}$ denotes the right derivative of the function $f$.Then,
	\begin{align*}
	&g(t) = \lim_{h\to 0^+}\frac{f(x_1,x_2,\ldots,x_{n-1},t+h)-f(x_1,x_2,\ldots,x_{n-1},t)}{h} \\
	&= \lim_{h\to 0^+}\frac{\max\{x_1,x_2,\ldots,x_{n-1},t+h\}-\max\{x_1,x_2,\ldots,x_{n-1},t\}}{h}\\
	& = \lim_{h\to 0^+}\frac{\max\{\max\{x_1,x_2,\ldots,x_{n-1}\},t+h\}-\max\{\max\{x_1,x_2,\ldots,x_{n-1}\},t\}}{h}.
	\end{align*}
	Suppose that  $\max\{x_1,x_2,\ldots,x_{n-1}\} = x_j$ for some $j=1,2,\ldots,n-1$. Then 
	\[
	g(t)  = \lim_{h\to 0^+}\frac{\max\{x_j,t+h\}-\max\{x_j,t\}}{h}.
	\]
	First, consider the case where $x_j >t$. Then, $h$ can be chosen sufficiently small such that $x_j > t+h$ and therefore $g(t) = 0$. For the case where $x_j<t$ and since $h>0$, we also have that $x_j<t+h$. Thus, $g(t) = 1$ which concludes the proof.	
\end{proof}

\section{Doob-type inequalities and related results}
We provide a Doob-type inequality for positive multiindexed demimartingales. The result is inspired by the work of Cairoli in \cite{CR1970} for multidimensionally indexed submartingales.

\begin{theorem}
	\label{theorem1}	Let $\left\{S_{\mathbf{n}}, \mathbf{n} \in \mathbb{N}^{k}\right\}$, be a positive multidimensionally indexed demimartingale. Then,
	\begin{enumerate}
		\item [a. ] For $p>1$ and $ES_{\mathbf{n}}<\infty$,
		$$E\left( \displaystyle \max_{\mathbf{i} \leq \mathbf{n}} S_{\mathbf{i}}\right)^p \leq \left(\frac{p}{p-1}\right)^{kp}E\left( S_{\mathbf{n}}\right)^p.$$
		\item [b. ] For any $\epsilon>0$ and $E\left( S_{\mathbf{n}}\left(\log^{+} S_{\mathbf{n}}\right)^k\right)<\infty$,
		$$ \epsilon P\left ( \displaystyle \max_{\mathbf{i} \leq \mathbf{n}} S_{\mathbf{i}} \geq \epsilon\right) \leq \sum_{i=1}^{k} (i-1)!A^i + k! A^k E\left( S_{\mathbf{n}}\left(\log^{+} S_{\mathbf{n}}\right)^k\right)$$
		where $A = \frac{e}{e-1}$.
	\end{enumerate}
\end{theorem}
\begin{proof}
	We start with the first inequality. The proof will follow by induction; we discuss first the case for $k=2$ i.e. the case of a positive 2-indexed demimartingale $\{S_{ij}, (i,j) \in\mathbb{N}^{2} \}$. For $p >1$, 
	\begin{equation}
	\label{moment1}	E\left(\max_{(i,j)\leq (n_1,n_2)}S_{ij}\right)^p  = E\left(\max_{1\leq j\leq n_2}\max_{1\leq i\leq n_1}S_{ij}\right)^p  = E\left(\max_{1\leq j\leq n_2}Y_j\right)^p
	\end{equation}
	where $Y_j = \displaystyle\max_{1\leq i\leq n_1}S_{ij}$ for $j=1,2,\ldots,n_2$. Let $f$ be a componentwise nondecreasing function. Then, 
	\begin{align*}
	&E[(Y_{j+1}-Y_j)f(Y_1,\ldots,Y_j)] = E\left[ \left(  \max_{1\leq i\leq n_1}S_{ij+1}  -\max_{1\leq i\leq n_1}S_{ij}      \right) f\left(\max_{1\leq i\leq n_1}S_{i1},\max_{1\leq i\leq n_1}S_{i2}\ldots, \max_{1\leq i\leq n_1}S_{ij}\right)\right]\\
	&\geq E\left(\sum_{i=1}^{n_1}(S_{ij+1}-S_{ij})g^{'}_{i+}(S_{1j},\ldots,S_{n_1j})    f\left(\max_{1\leq i\leq n_1}S_{i1},\max_{1\leq i\leq n_1}S_{i2}\ldots, \max_{1\leq i\leq n_1}S_{ij}\right)      \right)\\
	&=\sum_{i=1}^{n_1}E\left((S_{ij+1}-S_{ij})g^{'}_{i+}(S_{1j},\ldots,S_{n_1j})    f\left(\max_{1\leq i\leq n_1}S_{i1},\max_{1\leq i\leq n_1}S_{i2}\ldots, \max_{1\leq i\leq n_1}S_{ij}\right)      \right)\\
	&\geq 0
	\end{align*}
	where $g^{'}_{i+}$ denotes the right derivative of the function $g(x_1,\ldots,x_n) = \max\{x_1,\ldots,x_n\}$ with respect to its $i$-th component which is a nonnegative constant based on the previous lemma. The first inequality is due to the convexity of the maximum function while the last inequality is due to the demimartingale property of the sequence $\{S_{ij}, 1\leq j\leq n_2\}$. Thus, the sequence $Y_j = \displaystyle\max_{1\leq i\leq n_1}S_{ij}$ forms a demimartingale sequence. By combining the latter result with \eqref{moment1} and Theorem \ref{Wang} we have that
	\begin{align*}
	&	E\left(\max_{(i,j)\leq (n_1,n_2)}S_{ij}\right)^p  \leq \left( \frac{p}{p-1}\right)^pE(Y_{n_2})^p=  \left( \frac{p}{p-1}\right)^p E \left(\max_{1\leq i\leq n_1} S_{in_2}\right)^p \leq \left( \frac{p}{p-1}\right)^{2p}E(S_{n_1n_2})^p.
	\end{align*}
	Note that for obtaining the last inequality we applied Theorem \ref{Wang} again since $\{S_{in_2}, 1\leq i\leq n_1\}$ forms a single indexed demimartingale. We assume that the statement is true for $k-1$ and we consider a $k$-indexed demimartingale $\{S_{\mathbf{n}}, \mathbf{n}\in\mathbb{N}^{k}\}$. Let $(i,i_2,\ldots,i_k)=(i,\mathbf{s})$ where $\mathbf{s} = (i_2, \ldots,i_{k})$. Then, for a fixed $i, \, 1\leq i\leq n_1$, the sequence $\{S_{i\mathbf{s}}, \mathbf{s}\in \mathbb{N}^{k-1}\}$ forms a $(k-1)$-indexed demimartingale. Then,
	\begin{align*}
	&E\left( \max_{\mathbf{i} \leq \mathbf{n}} S_{\mathbf{i}}\right)^p  = E\left( \max_{\mathbf{s}\leq \mathbf{n^{*}}}\max_{1\leq i\leq n_1}S_{i\mathbf{s}}\right)^p = E\left( \max_{\mathbf{s}\leq \mathbf{n^{*}}}Y_{\mathbf{s}}\right)^p
	\end{align*}
	where $\mathbf{n^*} = (n_2, \ldots,n_k)$.
	By applying similar arguments as the ones for the case where $k=2$, we can prove that $Y_{\mathbf{s}}$ is a $(k-1)$-indexed demimartingale.  By the induction hypothesis and the result of Theorem \ref{Wang} we have that 
	\begin{align*}
	&E\left( \max_{\mathbf{i} \leq \mathbf{n}} S_{\mathbf{i}}\right)^p \leq \left( \frac{p}{p-1}\right)^{(k-1)p}E\left( \max_{1\leq i\leq n_1} S_{in_2\ldots n_k}\right)^p
	\leq \left (\frac{p}{p-1}\right)^{kp}E\left( S_{\mathbf{n}}\right)^p
	\end{align*}
	since $\{S_{in_2\ldots n_k}, 1\leq i\leq n_1\}$ forms a single indexed demisubmartingale.
	
	\medskip	
	
	\noindent For the second inequality, we start again with the case $k=2$. For any $\epsilon>0$,
	\begin{align*}
	\epsilon P\left( \max_{(i,j)\leq (n_1, n_2)} S_{ij} \geq \epsilon \right) \leq E\left( \max_{(i,j)\leq (n_1, n_2)} S_{ij}\right) = E\left( \max_{1\leq j\leq n_2} Y_j\right)
	\end{align*}
	where $Y_j = \displaystyle\max_{1\leq i\leq n_1}S_{ij}$. Recall that this is a demimartingale sequence and by employing the second inequality in Theorem \ref{Wang} we have	
	\begin{align*}
	& \epsilon P\left( \max_{(i,j)\leq (n_1, n_2)} S_{ij} \geq \epsilon \right) \leq E\left( \max_{(i,j)\leq (n_1, n_2)} S_{ij}\right) = E\left( \max_{1\leq j\leq n_2} Y_j\right)\\
	&\leq \frac{e}{e-1} +\frac{e}{e-1}E(Y_{n_2}\log^{+}Y_{n_2}) = A+AE\left( \max_{1\leq i\leq n_1}S_{in_2}\log^{+}\max_{1\leq i\leq n_1}S_{in_2}\right)\\
	&=A+A E\left( \max_{1\leq i\leq n_1}(S_{in_2}\log^{+}S_{in_2})\right)
	\end{align*}
	for $A = \frac{e}{e-1}$. Observe that the function $x\log^{+}x$ is nondecreasing convex and hence $(Z_i)_{i\geq 1} = (S_{in_2}\log^{+}S_{in_2})_{i\geq 1}$ is a single indexed demisubmartingale (Lemma 2.1 in \cite{C2000}). Then, 
	\begin{align*}
	& \epsilon P\left( \max_{(i,j)\leq (n_1, n_2)} S_{ij} \geq \epsilon \right) \leq A+A(A + AEZ_{n_1}\log^{+}Z_{n_1}) \\ &= A+A^2 +A^2 E\left( S_{n_1n_2}\log^{+}S_{n_1n_2}\log^{+}(S_{n_1n_2}\log^{+}S_{n_1n_2})\right)\\
	&\leq A+A^2+2A^2E(S_{n_1n_2}(\log^{+}S_{n_1n_2})^2)
	\end{align*}
	where the last inequality follows from the fact that for any $x>0$, $\log^{+}(x\log^{+}x) \leq 2\log x$. We assume that the statement is true for $k-1$  and we will prove its validity for $k$. Following the same notation as in the first part, for any $\epsilon >0$,
	\begin{align}
	&	\nonumber \epsilon P\left(  \max_{\mathbf{i} \leq \mathbf{n}} S_{\mathbf{i}} \geq \epsilon\right) \leq E\left( \max_{\mathbf{s}\leq \mathbf{n^{*}}}\max_{1\leq i\leq n_1}S_{i\mathbf{s}}\right) = E\left( \max_{\mathbf{s}\leq \mathbf{n^{*}}}Y_{\mathbf{s}}\right)\\
	&\label{ind}\leq A+\sum_{i=2}^{k-1}(i-1)!A^{i} +(k-1)!A^{k-1}E(Y_{\mathbf{n^{*}}}(\log^{+}Y_{\mathbf{n^{*}}})^{k-1})
	\end{align}
	due to the induction hypothesis. Following similar steps as before we have that
	\begin{align*}
	&E\left(Y_{\mathbf{n^{*}}}(\log^{+}Y_{\mathbf{n^{*}}})^{k-1}\right) = E\left( \max_{1\leq i\leq n_1}S_{i\mathbf{n^*}}(\log^{+}\max_{1\leq i\leq n_1}S_{i\mathbf{n^*}})^{k-1}\right)=E\left( \max_{1\leq i\leq n_1}S_{i\mathbf{n^*}}(\log^{+}S_{i\mathbf{n^*}})^{k-1}\right)\\
	&\leq A + AE Z_{n_1}\log^{+}Z_{n_1} = A+AE\left(S_{\mathbf{n}}(\log^{+}S_{\mathbf{n}})^{k-1}\log^{+}\left(S_{\mathbf{n}}(\log^{+}S_{\mathbf{n}})^{k-1} \right)\right)\\
	&\leq A+AkE(S_{\mathbf{n}}(\log^{+}S_{\mathbf{n}})^{k}).
	\end{align*}
	Note that the first inequality is due to the fact that $(Z_ i)_{i\geq 1} = (S_{i\mathbf{n^*}}(\log^{+}S_{i\mathbf{n^*}})^{k-1})_{i\geq 1}$ forms a single indexed demisubmartingale while for the second one the inequality $\log^{+}(x(\log^{+}x)^{k-1}) \leq k\log^{+}x, \, x>0$ is used. The latter expression together with \eqref{ind} lead to the desired result.
\end{proof}

\medskip

\medskip

Theorem \ref{theorem1} was obtained by following the ideas of the corresponding result for positive submartingales proved by Cairoli \cite{CR1970}. In his paper, Cairoli provided counterexamples showing that some classical inequalities for maximums of submartingales with discrete time are not valid in the case of submartingales with multi-dimensional time, including Doob's inequality. Despite Cairoli's counterexample, Doob's type inequality indeed has an extension to the case of multidimensional index as proved in \cite{CS1990} (see Corollary 2.4 there). Moreover, under specific conditions, the well-known Doob's inequality is also valid for a subclass of multiindexed submartingales (see for example Proposition 1.6 in \cite{W1986}). Although, in Theorem \ref{theorem1}, we proved that Cairoli's inequalities for multiindexed martingales are also valid for multiindexed demimartingales, it is of interest to check whether the classical Doob's inequality can be obtained for multiindexed demimartingales. We start by defining the concept of multidimensional rank orders.

\medskip

\noindent Let $\mathbf{n} = (n_1 \, n_2 \, \ldots \, n_k)$. We define the multiindexed rank orders $R_{\mathbf{n}}^{(j)}$ by
\[
R_{\mathbf{n}}^{(j)} = \begin{cases}
j\mbox{-th largest of } \{S_\mathbf{m}, \mathbf{m}\leq \mathbf{n}\} &\, \mbox{for} \, j\leq \prod_{i=1}^{k} n_i\\
\displaystyle\min_{\mathbf{m}\leq \mathbf{n}}S_\mathbf{m} & \, \mbox{for} \, j> \prod_{i=1}^{k} n_i.
\end{cases}
\]
The theorem that follows generalizes to the case of multiindexed demi(sub)martingales a result for single- indexed demjmartingles which can be found in \cite{NW1982}.

\begin{theorem}
	\label{Doob_new}Let $\left\{S_{\mathbf{n}}, \mathbf{n} \in \mathbb{N}^{k}\right\}$, be a multiindexed demi(sub)martingale with $S_{\boldsymbol{\ell}}=0$ when $\prod_{i=1}^{k} \ell_{i}=0$ and let $g$ be a (nonnegative) nondecreasing function on $\mathbb{R}$ with $g(0) = 0$. Then for any $\mathbf{n}$ and any $j$,
	\begin{equation}
	\label{NW1}E\left[ \int_{0}^{R_{\mathbf{n}}^{(j)}} u{\rm d}g(u)\right] \leq E(S_\mathbf{n}g(R_{\mathbf{n}}^{(j)}))
	\end{equation}
	and for any $\epsilon>0$,
	\begin{equation}
	\label{NW2} \epsilon P\left( R_{\mathbf{n}}^{(j)} \geq \epsilon\right) \leq E\left(S_\mathbf{n}I\left(R_{\mathbf{n}}^{(j)}\geq \epsilon\right)\right).
	\end{equation}
\end{theorem}
\begin{proof}
	For fixed $j, n_1, n_2, \ldots, n_{k-1}$ we set 
	\[
	Y_i = R^{(j)}_{n_1n_2\cdots n_{k-1}i} \quad\mbox{for}\quad1\leq i\leq n_k
	\]
	with $Y_0 = 0$. First observe that
	\[
	S_{\mathbf{n}}g(R^{(j)}_{\mathbf{n}}) = \sum_{i=0}^{n_k-1}S_{\mathbf{n}^*,i+1}(g(Y_{i+1})-g(Y_{i}))+\sum_{i=1}^{n_k-1}(S_{\mathbf{n}^*,i+1}-S_{\mathbf{n}^*,i})g(Y_i)
	\]
	where $\mathbf{n}^* = (n_1 \, n_2\, \ldots\, n_{k-1}) $ and $S_{n^*,i} = S_{n_1 \, n_2\, \ldots\, n_{k-1} \,i}$.
	We want to prove that 
	\begin{equation}
	\label{ineq4}S_{\mathbf{n}^*,i+1}(g(Y_{i+1})-g(Y_{i}))\geq Y_{i+1}(g(Y_{i+1})-g(Y_{i})).
	\end{equation} 
	Note that 
	\[
	Y_{i+1}\geq Y_{i}
	\]
	where the non-degenerate case of \eqref{ineq4} is the one for which $Y_{i+1} = S_{\mathbf{n}^*,i+1}$ for $j=1$  while $Y_{i+1} > S_{\mathbf{n}^*,i+1}$ for any $j>1$ and  by taking into account the monotonicity of $g$, \eqref{ineq4} holds true for any $i$. Moreover, for any $i$,
	\[
	S_{\mathbf{n}^*,i+1}(g(Y_{i+1})-g(Y_{i}))\geq Y_{i+1}(g(Y_{i+1})-g(Y_{i}))\geq \int_{Y_i}^{Y_{i+1}}u{\rm d}g(u)
	\]
	which leads to 
	\[
	S_{\mathbf{n}}g(R^{(j)}_{\mathbf{n}})  \geq \int_{0}^{R^{(j)}_{\mathbf{n}}}u{\rm d}g(u)+\sum_{i=1}^{n_k-1}(S_{\mathbf{n}^*,i+1}-S_{\mathbf{n}^*,i})g(Y_i)
	\]
	By taking expectations on both sides we have that
	\[
	E[	S_{\mathbf{n}}g(R^{(j)}_{\mathbf{n}}) ] \geq  E\left[\int_{0}^{R^{(j)}_{\mathbf{n}}}u{\rm d}g(u)\right]+\sum_{i=1}^{n_k-1}E[(S_{\mathbf{n}^*,i+1}-S_{\mathbf{n}^*,i})g(Y_i)].
	\]
	The desired result follows by noticing that the last term is nonnegative due to the single-index demi(sub)martingale property of the sequence $\{S_{\mathbf{n}^*,i},\, 1\leq i\leq n_k\}$. Inequality \eqref{NW2} follows from \eqref{NW1} by choosing $g(u) = I(u\geq \epsilon)$.
\end{proof}

\medskip

\noindent As a direct consequence of Theorem \ref{Doob_new} we get the following inequalities. The single-index analogues can be found in \cite{PR2012} (see relations (2.7.1) and (2.7.2)).

\begin{corollary}
	Let $\left\{S_{\mathbf{n}}, \mathbf{n} \in \mathbb{N}^{k}\right\}$ be a multiindexed demimartingale with $S_{\boldsymbol{\ell}}=0$ when $\prod_{i=1}^{k} \ell_{i}=0$. Then for any $\epsilon >0$, 
	\begin{equation}
	\label{Doobv2}\epsilon P\left(\max_{\mathbf{i}\leq \mathbf{n}}S_{\mathbf{j}}\geq \epsilon\right) \leq E\left(S_\mathbf{n}I\left(\max_{\mathbf{i}\leq \mathbf{n}}S_{\mathbf{j}}\geq \epsilon\right)\right) 
	\end{equation}
	and
	\[
	\epsilon P\left(\min_{\mathbf{i}\leq \mathbf{n}}S_{\mathbf{i}}\geq \epsilon\right) \leq E\left(S_\mathbf{n}I\left(\min_{\mathbf{i}\leq \mathbf{n}}S_{\mathbf{i}}\geq \epsilon\right)\right) 
	\]
\end{corollary}

\medskip

	\begin{remark}
		As it has already been mentioned, for the case of multiindexed martingales the classical Doob's inequality cannot be established in general. However, inequality \eqref{Doobv2} shows that in the case of multiindexed demimartingales this celebrated result holds true. It is also important to highlight that, compared to Theorem \ref{theorem1}, the upper bound in \eqref{Doobv2} does not depend on the dimension of the index. Moreover, note that for $k=1$, inequality \eqref{Doobv2} is reduced to the result of Theorem \ref{Doob} for the case where $g(x) = x$.
	\end{remark}

\noindent The Doob-type inequality obtained in \eqref{Doobv2} becomes the source result for  various moment inequalities which generalize results that are already known for the case of single-indexed demimartingales.

\begin{corollary}
	\label{mom1}	Let $\left\{S_{\mathbf{n}}, \mathbf{n} \in \mathbb{N}^{k}\right\}$ be a nonnegative multidimensionally indexed demimartingale with $S_{\boldsymbol{\ell}}=0$ when $\prod_{i=1}^{k} \ell_{i}=0$. Then, 
	\[
	E\left(\max_{\mathbf{i}\leq \mathbf{n}}S_{\mathbf{j}}\right)^p \leq \begin{cases}
	\left(\frac{p}{p-1}\right)^pE(S_{\mathbf{n}})^p, &\quad\mbox{for}\quad p>1\\
	\frac{e}{e-1}+\frac{e}{e-1}ES_{\mathbf{n}}\log^+S_{\mathbf{n}}, &\quad\mbox{for}\quad p=1.
	\end{cases}
	\]
\end{corollary}

	\begin{remark}
		Observe that the inequalities of Corollary \ref{mom1}  provide sharper bounds compared to the ones obtained in Theorem \ref{theorem1} while for the case $k=1$, Corollary \ref{mom1} is reduced to Theorem \ref{Wang}.
	\end{remark}

\noindent For single-indexed positive martingales, Harremo\"{e}s in \cite{H2008} provided a maximal moment inequality which can be consider as a strengthening of a classical maximal inequality by Doob.   Prakasa Rao in \cite{PR2007} proved that Harremo\"{e}s result is also valid for a sequence of single-indexed positive demimartingales (see Theorem 2.7.3 in \cite{PR2012}). Motivated by these results, we provide a generalization to the case of multiindexed demimartingales. The key result for obtaining the particular moment inequality is again expression \eqref{Doobv2}. 

\begin{theorem}
	Let $\left\{S_{\mathbf{n}}, \mathbf{n} \in \mathbb{N}^{k}\right\}$ be a positive multidimensionally indexed demimartingale with $S_{\boldsymbol{\ell}}=0$ when $\prod_{i=1}^{k} \ell_{i}=0$ and $S_{\boldsymbol{\ell}}=c\in (0,1]$ when $\sum_{i=1}^{k} \ell_{i}=1$. Let $\gamma(x) = x-\ln x -c$ for $x>0$. Then,
	\[
	\gamma\left(  E\left( \max_{\mathbf{i}\leq \mathbf{n}}S_{\mathbf{i}}    \right)      \right) \leq 1-c^2-\ln c + ES_{\mathbf{n}}\ln S_{\mathbf{n}}.
	\]
	For the special case where $c=1$,
	\[
	\gamma\left(  E\left( \max_{\mathbf{i}\leq \mathbf{n}}S_{\mathbf{i}}    \right)      \right) \leq ES_{\mathbf{n}}\ln S_{\mathbf{n}}.
	\]
\end{theorem}
\begin{proof}
	We use the notation $S_{\mathbf{n}}^{\tiny\mbox{max}} = 	\displaystyle\max_{\mathbf{i}\leq \mathbf{n}}S_{\mathbf{i}}$. Then, 
	\begin{align*}
	& ES_{\mathbf{n}}^{\tiny\mbox{max}} -c = \int_{0}^{\infty}P\left( S_{\mathbf{n}}^{\tiny\mbox{max}} \geq x\right) {\rm d}x -c = \int_{0}^{c}P\left( S_{\mathbf{n}}^{\tiny\mbox{max}} \geq x\right) {\rm d}x+\int_{c}^{\infty}P\left( S_{\mathbf{n}}^{\tiny\mbox{max}} \geq x\right) {\rm d}x -c\\
	& \leq \int_{c}^{\infty}P\left( S_{\mathbf{n}}^{\tiny\mbox{max}} \geq x\right) {\rm d}x \leq  \int_{c}^{\infty} \left( \frac{1}{x} \int_{\{S_{\mathbf{n}}^{\tiny\mbox{max}} \geq x\}} S_{\mathbf{n}}{\rm d} P\right) {\rm d}x \qquad(\mbox{due to}\, \, \eqref{Doobv2})\\
	&= E\left(   S_{\mathbf{n}} \int_{c}^{S_{\mathbf{n}}^{\tiny\mbox{max}}}    \frac{1}{x}{\rm d}x\right) = E\left(   S_{\mathbf{n}}  \ln S_{\mathbf{n}}^{\tiny\mbox{max}}\right)-\ln c ES_{\mathbf{n}}.
	\end{align*}	
	Observe that due to the demimartingale property $ES_{\mathbf{n}} = ES_{\boldsymbol{\ell}} = c $ for $\sum_{i=1}^{k}\ell_i  =1$. Hence,
	\[
	ES_{\mathbf{n}}^{\tiny\mbox{max}} -c\leq E\left(   S_{\mathbf{n}} \ln S_{\mathbf{n}}^{\tiny\mbox{max}}\right)-c\ln c.
	\]
	It is known that $\forall \, x>0, \, \ln x \leq x-1 \, $. Since $c\in(0,1]$ we have that 	$\forall \, x>0, \,\ln x \leq x-c $ and hence $\gamma(x) \geq 0$ for $x>0$ and $c\in (0,1]$. Then,
	\begin{align*}
	& ES_{\mathbf{n}}^{\tiny\mbox{max}} -c\leq E\left[   S_{\mathbf{n}}   \left(   \ln S_{\mathbf{n}}^{\tiny\mbox{max}} +\gamma\left(    \frac{S_{\mathbf{n}}^{\tiny\mbox{max}}}{S_{\mathbf{n}} ES_{\mathbf{n}}^{\tiny\mbox{max}}}                   \right)              \right)          \right] - c\ln c
	\leq 1 + ES_{\mathbf{n}}\ln S_{\mathbf{n}}+ c(\ln ES_{\mathbf{n}}^{\tiny\mbox{max}} -c -\ln c).
	\end{align*}
	Now,
	\begin{align*}
	&\gamma\left(ES_{\mathbf{n}}^{\tiny\mbox{max}}\right) = ES_{\mathbf{n}}^{\tiny\mbox{max}} - c- \ln ES_{\mathbf{n}}^{\tiny\mbox{max}}
	\leq 1 + ES_{\mathbf{n}}\ln S_{\mathbf{n}}+ c(\ln ES_{\mathbf{n}}^{\tiny\mbox{max}} -c -\ln c)- \ln ES_{\mathbf{n}}^{\tiny\mbox{max}}\\
	&\leq 1-c^2-c\ln c +(c-1)\ln E\left(\max_{\mathbf{i}\leq \mathbf{n}}S_{\mathbf{i}}\right)+ES_{\mathbf{n}}\ln S_{\mathbf{n}}. 
	\end{align*}
	The desired result follows by noticing that $\displaystyle\max_{\mathbf{i}\leq \mathbf{n}}S_{\mathbf{i}}\geq c$ and since $c\in(0,1]$, $$(1-c)\ln E\left(\max_{\mathbf{i}\leq \mathbf{n}}S_{\mathbf{i}}\right) \geq (1-c)\ln c.$$
\end{proof}

\medskip

\noindent Next, we obtain a moment inequality for positive multiindexed demimartingales which is motivated by Theorem 3.1 in \cite{WHYS2011}. 

\begin{theorem}
	Let $\left\{S_{\mathbf{n}}, \mathbf{n} \in \mathbb{N}^{k}\right\}$ be a positive multiindexed demimartingale with $S_{\boldsymbol{\ell}}=0$ when $\prod_{i=1}^{k} \ell_{i}=0$ and $S_{\boldsymbol{\ell}}=c>0$ when $\sum_{i=1}^{k} \ell_{i}=1$. Moreover, assume that $ES_{\mathbf{n}}\ln S_{\mathbf{n}} <\infty, \, \forall \mathbf{n} \in \mathbb{N}^{k}$ and $\displaystyle\lim_{\mathbf{n}\to\infty} ES_{\mathbf{n}}\ln S_{\mathbf{n}} = \infty$. Then,
	\[
	\limsup_{\mathbf{n}\to\infty} \dfrac{E(\max_{\mathbf{j}\leq \mathbf{n}}S_{\mathbf{j}})}{ES_{\mathbf{n}}\ln S_{\mathbf{n}}}\leq 1.
	\]
\end{theorem}
\begin{proof}
	It was proven earlier that
	\[
	ES_{\mathbf{n}}^{\tiny\mbox{max}} -c\leq E\left(   S_{\mathbf{n}} \ln S_{\mathbf{n}}^{\tiny\mbox{max}}\right)-c\ln c.
	\]
	According to \cite{WHYS2011} (see relation (3.2) there) for any $a,b>0$ and $x_0>e$
	\[
	b\ln a\leq b\ln b+ax_0^{-1}+b(\ln x_0 -1).
	\]
	By employing this inequality we have that
	\begin{align*}
	&ES_{\mathbf{n}}^{\tiny\mbox{max}} -c\leq E\left(   S_{\mathbf{n}} \ln S_{\mathbf{n}}\right)-c\ln c +x_0^{-1}ES_{\mathbf{n}}^{\tiny\mbox{max}} + ES_{\mathbf{n}}(\ln x_0-1).
	\end{align*}
	Recall that $ES_{\mathbf{n}} = c$ and therefore, after some algebraic calculations, we get that
	\[
	\dfrac{ES_{\mathbf{n}}^{\tiny\mbox{max}}}{ES_{\mathbf{n}}\ln S_{\mathbf{n}}} \leq \dfrac{x_0}{x_0-1}\left(   1 + \dfrac{c(\ln x_0 - \ln c)}{ES_{\mathbf{n}}\ln S_{\mathbf{n}}}       \right).
	\]
	The result is obtained by taking $\limsup$ on both sides as $\mathbf{n} \to \infty$ and then let $x_0$ to tend to infinity.
\end{proof}

\section{Chow-type maximal inequality}
Hadjikyriakou in \cite{MH2010} proved the following Chow-type inequality for multiindexed demimartingales (see also \cite{CH2011}). The result was used to obtain an asymptotic result and further maximal inequalities.
\begin{theorem}
Let $\left\{S_{\mathbf{n}}, \mathbf{n} \in \mathbb{N}^k\right\}$ be a multiindexed demimartingale with $S_{\boldsymbol{\ell}}=0$ when $\prod_{i=1}^{k} \ell_{i}=0$ and $\left\{c_{\mathbf{n}}, \mathbf{n} \in\mathbb{N}^k\right\}$ be a nonincreasing array of positive numbers. Let $g$ be a nonnegative convex function on $\mathbb{R}$ with $g(0)=0$.
Then, $\forall \varepsilon>0$,
$$
\varepsilon P\left(\max_{\mathbf{j} \leq \mathbf{n}}c_{\mathbf{j}} g\left(S_{\mathbf{j}}\right) \geq \varepsilon\right) \leq \min _{1 \leq s \leq k}\left\{\sum_{\mathbf{j} \leq \mathbf{n}} c_{\mathbf{j}} E\left[g\left(S_{\mathbf{j} ; s ; i}\right)-g\left(S_{\mathbf{j} ; s ; i-1}\right)\right]\right\}.
$$
\end{theorem}

\medskip

Next, we provide a Chow-type inequality for multiindexed demimartingales by applying the methodology introduced by \cite{W2004}. The new approach leads to an upper bound which depends only on a single summation.
\begin{theorem}
\label{Chownew}Let $\left\{S_{\mathbf{n}}, \mathbf{n} \in \mathbb{N}^{k}\right\}$ be a multiindexed demimartingale with $S_{\boldsymbol{\ell}}=0$ when $\prod_{i=1}^{k} \ell_{i}=0$. Let $g$ be a nonnegative convex function on $\mathbb{R}$ with $g(0) = 0$ and $\{c_{\mathbf{n}}, \mathbf{n} \in \mathbb{N}^{k}\}$ be a nonincreasing array of positive numbers. Then, for all $\epsilon >0$
	\[
	\epsilon P(\max_{\mathbf{j}\leq \mathbf{n}} c_\mathbf{j}g(S_\mathbf{j})\geq \epsilon)\leq \min_{1\leq s\leq k}\sum_{i=1}^{n_s}c_{\mathbf{n};s;i}E\left[ \left( g(S_{\mathbf{n};s;i}) -   g(S_{\mathbf{n};s;i-1})   \right) I\left(\max_{\mathbf{k}\leq \mathbf{n}} c_\mathbf{k}g(S_\mathbf{k})\geq \epsilon\right)  \right]
	\]
\end{theorem}
\begin{proof}
	The proof is motivated by the the proof of Theorem 2.1 in \cite{W2004}. First, we define the functions
	\[
	u(x) = g(x) I\{x\geq 0\}\quad\mbox{and}\quad v(x) = g(x)I\{x<0\}
	\]
	which are both nonnegative convex functions with $u(x)$ being nondecreasing while $v(x)$ is a nonincreasing function. Observe that $g(x) = u(x)+v(x) = \max\{u(x), v(x)\}$. Therefore,
	\begin{equation}
	\label{chowang1}\epsilon P(\max_{\mathbf{j}\leq \mathbf{n}} c_\mathbf{j}g(S_\mathbf{j})\geq \epsilon)\leq \epsilon P(\max_{\mathbf{j}\leq \mathbf{n}} c_\mathbf{j}u(S_\mathbf{j})\geq \epsilon)+\epsilon P(\max_{\mathbf{j}\leq \mathbf{n}} c_\mathbf{j}v(S_\mathbf{j})\geq \epsilon).
	\end{equation}
	Following the steps and the notation introduced by \cite{W2004}, we consider $m(\cdot)$ to be a nonnegative nondecreasing function with $m(0) = 0$ and let 
	\[
	S'_{\mathbf{n}} = \max_{\mathbf{j}\leq \mathbf{n}} c_\mathbf{j}u(S_\mathbf{j}).
	\]
	Without loss of generality we fix $n_1, \ldots, n_{k-1}$ and denote $(\mathbf{n}^*,i) = (n_1 \, \ldots \,  n_{k-1} \, i)$. Then,
	\begin{align*}
	&E\left[ \int_{0}^{S'_{\mathbf{n}}}   t{\rm d}m(t)   \right] = \sum_{i=1}^{n_k}E\left[ \int_{S'_{\mathbf{n^*},i-1}}^{S'_{\mathbf{n}^*,i}}   t{\rm d}m(t)   \right]\leq \sum_{i=1}^{n_k} E\left[ S'_{\mathbf{n^*},i}           \left(  m\left(    S'_{\mathbf{n^*},i}      \right)    -      m\left(    S'_{\mathbf{n^*},i-1}      \right)               \right)\right]\\
	&\leq \sum_{i=1}^{n_k} c_{\mathbf{n^*},i}E\left[ u(S_{\mathbf{n^*},i})           \left(  m\left(    S'_{\mathbf{n^*},i}      \right)    -      m\left(    S'_{\mathbf{n^*},i-1}      \right)               \right)\right]\leq  \sum_{i=1}^{n_k} c_{\mathbf{n^*},i}E\left[ m(S'_{\mathbf{n}})           \left(  u\left(    S_{\mathbf{n^*},i}      \right)    -      u\left(    S_{\mathbf{n^*},i-1}      \right)               \right)\right]-A
	\end{align*}
	where 
	\[
	A = \sum_{i=1}^{n_k-1}E \left[ (c_{\mathbf{n^*},i+1} u(S_{\mathbf{n^*},i+1}) -  c_{\mathbf{n^*},i} u(S_{\mathbf{n^*},i})  ) m(S'_{\mathbf{n^*},i})                      \right]+\sum_{i=1}^{n_k-1}( c_{\mathbf{n^*},i}-c_{\mathbf{n^*},i+1}) E[u(S_{\mathbf{n^*},i})m(S'_{\mathbf{n}}) ].
	\]
	The second inequality follows by observing that 
	$ S'_{\mathbf{n^*},i}\geq S'_{\mathbf{n^*},i-1}$ and thus $S'_{\mathbf{n^*},i} = c_{\mathbf{n^*},i}u(S_{\mathbf{n^*},i})$ or $m(    S'_{\mathbf{n^*},i})    = m(    S'_{\mathbf{n^*},i-1} ) $. We need to prove that $A$ is a nonnegative term. Note that due to the fact that $( c_{\mathbf{n^*},i}-c_{\mathbf{n^*},i+1})u(S_{\mathbf{n^*},i})\geq 0$ for any $i$, and by the convexity of the function $u(\cdot)$ we have that 
	\begin{align*}
	&A \geq \sum_{i=1}^{n_k-1}E \left[ (c_{\mathbf{n^*},i+1} u(S_{\mathbf{n^*},i+1}) -  c_{\mathbf{n^*},i} u(S_{\mathbf{n^*},i})  ) m(S'_{\mathbf{n^*},i})                      \right]+\sum_{i=1}^{n_k-1}( c_{\mathbf{n^*},i}-c_{\mathbf{n^*},i+1}) E[u(S_{\mathbf{n^*},i})m(S'_{\mathbf{n^*},i}) ]\\
	&= \sum_{i=1}^{n_k-1}  c_{\mathbf{n^*},i+1}E\left[  (    u(S_{\mathbf{n^*},i+1}) -      u(S_{\mathbf{n^*},i})    )      m(S'_{\mathbf{n^*},i})                    \right]\geq \sum_{i=1}^{n_k-1}  c_{\mathbf{n^*},i+1}E\left[  (    S_{\mathbf{n^*},i+1} -      S_{\mathbf{n^*},i})    h(S_{\mathbf{n^*},i})  m(S'_{\mathbf{n^*},i})                    \right]\geq 0
	\end{align*}
	since $  h(S_{\mathbf{n^*},i})  m(S'_{\mathbf{n^*},i})$ is a nondecreasing function of $\{S_{\mathbf{n^*},i}, 1\leq i\leq n_k\}$ which forms a single indexed demimartingale. Consider the case where $m(t) = I\{ t\geq \epsilon \}$. Then, since $S'_{\mathbf{n}} \leq \displaystyle \max_{\mathbf{j}\leq \mathbf{n}} c_\mathbf{j}g(S_\mathbf{j})$
	\begin{align}
	&\nonumber\epsilon P(\max_{\mathbf{j}\leq \mathbf{n}} c_\mathbf{j}u(S_\mathbf{j})\geq \epsilon) \leq \sum_{i=1}^{n_k} c_{\mathbf{n^*},i}E\left[           \left(  u\left(    S_{\mathbf{n^*},i}      \right)    -      u\left(    S_{\mathbf{n^*},i-1}      \right)               \right)I\{S'_{\mathbf{n}}\geq \epsilon\} \right]\\
	&\nonumber= \sum_{i=1}^{n_k-1}( c_{\mathbf{n^*},i} -c_{\mathbf{n^*},i+1}) E[u(S_{\mathbf{n^*},i})I\{S'_{\mathbf{n}}\geq \epsilon\}] +c_{\mathbf{n}}E[u\left(    S_{\mathbf{n}  }    \right)I\{S'_{\mathbf{n}}\geq \epsilon\}]\\
	&\nonumber\leq  \sum_{i=1}^{n_k-1}( c_{\mathbf{n^*},i} -c_{\mathbf{n^*},i+1}) E[u(S_{\mathbf{n^*},i})I\{\max_{\mathbf{j}\leq \mathbf{n}} c_\mathbf{j}g(S_\mathbf{j})\geq \epsilon\}] +c_{\mathbf{n}}E[u\left(    S_{\mathbf{n}  }    \right)I\{\max_{\mathbf{j}\leq \mathbf{n}} c_\mathbf{j}g(S_\mathbf{j})\geq \epsilon\}]\\
	&\label{chowwangu}=\sum_{i=1}^{n_k} c_{\mathbf{n^*},i}E\left[           \left(  u\left(    S_{\mathbf{n^*},i}      \right)    -      u\left(    S_{\mathbf{n^*},i-1}      \right)               \right)I\{\max_{\mathbf{j}\leq \mathbf{n}} c_\mathbf{j}g(S_\mathbf{j})\geq \epsilon\} \right].
	\end{align}
	Similarly, it can be proven that 
	\begin{equation}
	\label{chowwangv}\epsilon P(\max_{\mathbf{j}\leq \mathbf{n}} c_\mathbf{j}v(S_\mathbf{j})\geq \epsilon) \leq \sum_{i=1}^{n_k} c_{\mathbf{n^*},i}E\left[           \left(  v\left(    S_{\mathbf{n^*},i}      \right)    -      v\left(    S_{\mathbf{n^*},i-1}      \right)               \right)I\{\max_{\mathbf{j}\leq \mathbf{n}} c_\mathbf{j}g(S_\mathbf{j})\geq \epsilon\} \right].
	\end{equation}
	By combining \eqref{chowang1}-\eqref{chowwangv}, we have that 
	\[
	\epsilon P(\max_{\mathbf{j}\leq \mathbf{n}} c_\mathbf{j}g(S_\mathbf{j})\geq \epsilon)\leq \sum_{i=1}^{n_k} c_{\mathbf{n^*},i}E\left[           \left(  g\left(    S_{\mathbf{n^*},i}      \right)    -      g\left(    S_{\mathbf{n^*},i-1}      \right)               \right)I\{\max_{\mathbf{j}\leq \mathbf{n}} c_\mathbf{j}g(S_\mathbf{j})\geq \epsilon\} \right]
	\]
	which leads to the desired result.
\end{proof}

\medskip

\noindent As a direct consequence of the Chow-type inequality proven above, we can easily obtain the following convergence result.
\begin{theorem}
Let $\left\{S_{\mathbf{n}}, \mathbf{n} \in \mathbb{N}^{k}\right\}$ be a multiindexed demimartingale with $S_{\boldsymbol{\ell}}=0$ when $\prod_{i=1}^{k} \ell_{i}=0$. Let $g$ be a nonnegative convex function on $\mathbb{R}$ with $g(0) = 0$ and $\{c_{\mathbf{n}}, \mathbf{n} \in \mathbb{N}^{k}\}$ be a nonincreasing array of positive numbers. Assume that, for some $1\leq s\leq k$ and $p\geq 1$, $$\sum_{i=1}^{\infty}c^p_{\mathbf{n};s;i}E([g(S_{\mathbf{n};s;i})]^p-[g(S_{\mathbf{n};s;i-1})]^p)<\infty\quad\mbox{and}\quad c^p_{\mathbf{n}}E([g(S_{\mathbf{n}})]^p-[g(S_{\mathbf{n};s;n_s-1})]^p)\to 0,\,  n\to \infty.$$ Then,
	\[
	c_{\mathbf{n}}g(S_{\mathbf{n}})\to0\quad\mbox{a.s.,}\quad\mathbf{n}\to \infty.
	\]
\end{theorem}
\begin{proof}
	Without loss of generality, we assume that the conditions are satisfied for $s=k$ and let $\mathbf{N} = (N \, N \, \ldots \, N)$. Then, we can write
\begin{align*}
&P(\max_{\mathbf{n}\geq \mathbf{N}} c_\mathbf{n}g(S_\mathbf{n})\geq \epsilon) = P(\max_{\mathbf{n}\geq \mathbf{N}} c^p_\mathbf{n}[g(S_\mathbf{n})]^p\geq \epsilon^p) \\
& \leq \frac{1}{\epsilon^p}\sum_{i\geq N}c^p_{\mathbf{N^*},i}E\left[           \left(  [g\left(    S_{\mathbf{N^*},i}      \right) ]^p   -      [g\left(    S_{\mathbf{N^*},i-1}      \right)]^p               \right)I\{\max_{\mathbf{n}\geq \mathbf{N}} c^p_\mathbf{n}[g(S_\mathbf{n})]^p\geq \epsilon\} \right]\\
& \leq \frac{1}{\epsilon^p}\sum_{i\geq N}c^p_{\mathbf{N^*},i}E\left[           \left(  [g\left(    S_{\mathbf{N^*},i}      \right) ]^p   -      [g\left(    S_{\mathbf{N^*},i-1}      \right)]^p               \right) \right]\\
&= \frac{1}{\epsilon^p}c^p_{\mathbf{N}}E(([g(S_{\mathbf{N}})]^p - [g(S_{\mathbf{N^*},N-1})]^p) ) + \frac{1}{\epsilon^p}\sum_{i= N+1}^{\infty}c^p_{\mathbf{N^*},i}E\left[           \left(  [g\left(    S_{\mathbf{N^*},i}      \right)] ^p   -      [g\left(    S_{\mathbf{N^*},i-1}    \right)     ]^p          \right) \right]\to 0
\end{align*}
as $N\to\infty$.
\end{proof}

\medskip

\noindent The Chow-type inequality provided above can lead to a H\'{a}jek-R\'{e}nyi inequality for multiindexed associated random variables.

\begin{corollary}
	Let $\left\{X_{\mathbf{n}}, \mathbf{n} \in \mathbb{N}^k\right\}$ be mean zero multiindexed associated random variables, $\left\{c_{\mathbf{n}}, \mathbf{n} \in \mathbb{N}^d\right\}$ a nonincreasing array of positive numbers and $S_{\mathbf{n}}=\sum_{\mathbf{i} \leq \mathbf{n}} X_{\mathbf{i}}$. Then $\forall \varepsilon>0$
	\[
	P(\max_{\mathbf{i} \leq \mathbf{n}}c_{\mathbf{i}}|S_\mathbf{i}|\geq \epsilon)\leq \min_{1 \leq s \leq k} \epsilon^{-2}\sum_{j=1}^{n_s}c_{\mathbf{n};s;j}[2\Cov(X_{\mathbf{n};s;j},S_{\mathbf{n};s;j-1})+E(X^2_{\mathbf{n};s;j})]
	\]
\end{corollary}
\begin{proof}
	\begin{align*}
	& P(\max_{\mathbf{j}\leq \mathbf{n}} c_\mathbf{j}|S_\mathbf{j}|\geq \epsilon) = P(\max_{\mathbf{j}\leq \mathbf{n}} c^2_\mathbf{j}|S_\mathbf{j}|^2)\geq \epsilon^2)
	\leq \min_{1\leq s\leq k}\sum_{i=1}^{n_s}c^2_{\mathbf{n};s;i}E\left[ \left( |S_{\mathbf{n};s;i}|^2 -   |S_{\mathbf{n};s;i-1}|^2   \right) I\left(\max_{\mathbf{k}\leq \mathbf{n}} c^2_\mathbf{k}|S_\mathbf{k}|^2\geq \epsilon^2\right)  \right]\\
	&\leq \min_{1\leq s\leq k}\sum_{i=1}^{n_s}c^2_{\mathbf{n};s;i}E\left[ \left( |S_{\mathbf{n};s;i}|^2 -   |S_{\mathbf{n};s;i-1}|^2   \right)   \right]\leq \min_{1\leq s\leq k}\sum_{i=1}^{n_s}c^2_{\mathbf{n};s;i}E\left[ \left( S_{\mathbf{n};s;i} - S_{\mathbf{n};s;i-1}   \right) \left( S_{\mathbf{n};s;i} + S_{\mathbf{n};s;i-1}   \right)   \right]\\
	&= \min_{1\leq s\leq k}\sum_{i=1}^{n_s}c^2_{\mathbf{n};s;i}E\left[   X_{\mathbf{n};s;i} (2S_{\mathbf{n};s;i-1} +X_{\mathbf{n};s;i})            \right]
	\end{align*}
	which gives the desired result.
\end{proof}

\begin{remark}
	Note that the Chow-type inequality obtained in \cite{MH2010} (or in \cite{CH2011}), also lead to a convergence result and  a H\'{a}jek-R\'{e}nyi inequality for multiindexed associated random variables. It is important to highlight however, that the results obtained here involve a sum over a single index while in \cite{MH2010} the conditions for the asymptotic result and the upper bound of the H\'{a}jek-R\'{e}nyi inequality depend on multiple summations.
\end{remark}

\section{Maximal $\phi$-inequalities for nonnegative multiindexed demisubmartingales}
Let $\mathcal{C}$ denote the class of Orlicz functions i.e. unbounded, nondecreasing convex functions $\phi:[0, \infty) \rightarrow[0, \infty)$ with $\phi(0)=0$. Given $\phi \in \mathcal{C}$ and $a \geq 0$, define
$$
\Phi_a(x)=\int_a^x \int_a^s \frac{\phi^{\prime}(r)}{r} d r d s, \quad x>0
$$
and
$$
p_\phi^*=\sup _{x>0} \frac{x \phi^{\prime}(x)}{\phi(x)}.
$$
Note that the function $\phi$ is called moderate if $p_\phi^*<\infty$. More information on Orlicz functions can be found in \cite{AR2006}.

\medskip

\noindent The single index analogues of the results that follow can be found in \cite{PR2007} (or see Section 2.8 in \cite{PR2012}). 
\begin{theorem}
	\label{phiTheorem1}Let $\left\{S_{\mathbf{n}}, \mathbf{n} \in \mathbb{N}^{k}\right\}$ be a nonnegative multiindexed demisubmartingale with $S_{\boldsymbol{\ell}}=0$ when $\prod_{i=1}^{k} \ell_{i}=0$ and let $\phi \in \mathcal{C}$. Then, for $x>0$ and $\lambda\in (0,1)$,
	\begin{equation}
	\label{OrlProb} P\left( \max_{\mathbf{i} \leq \mathbf{n}} S_{\mathbf{i}}\geq x \right) \leq \dfrac{\lambda}{(1-\lambda)x}\int_{x}^{\infty}P(S_{\mathbf{n}}>\lambda y){\rm d}y = \dfrac{\lambda}{(1-\lambda)x} E\left( \dfrac{S_{\mathbf{n}}}{\lambda}-x\right)^{+}.
	\end{equation}
	Moreover,
	\begin{equation}
	\label{OrlExp}
	E\left(\phi \left(       \max_{\mathbf{i} \leq \mathbf{n}} S_{\mathbf{i}}\    \right)\right)\leq \phi(b)+ \dfrac{\lambda}{1-\lambda}\int_{[S_\mathbf{n}>\lambda b]}\left(  \Phi_a \left(\frac{S_\mathbf{n}}{\lambda}\right)-\Phi_a(b)-\Phi'(b)\left(\frac{S_\mathbf{n}}{\lambda}-b\right) \right){\rm d}P
	\end{equation}
	for all $\mathbf{n}\in \mathbb{N}^k$, $a,b>0$. If $\phi'(x)/x$ is integrable at 0, then the latter inequality holds for $b=0$.	
\end{theorem}
\begin{proof}
	First observe that by \eqref{Doobv2} we have that, for any $x>0$,
	\[
	P\left( \max_{\mathbf{i} \leq \mathbf{n}} S_{\mathbf{i}}\geq x \right) \leq \frac{1}{x}E\left(S_{\mathbf{n}}I\left(  \max_{\mathbf{i} \leq \mathbf{n}} S_{\mathbf{i}}\geq x   \right)\right) = \frac{1}{x}\int_{0}^{\infty}P(S_{\mathbf{n}} \geq y,S_{\mathbf{n}}^{\tiny\mbox{max}} \geq x ){\rm d}y.
	\]
	The rest of the proof runs along the same lines as in the case of a single index (see Theorem 2.8.1 in \cite{PR2012}).
\end{proof}

\medskip

\noindent The moment inequality presented in \eqref{OrlExp} becomes the source result for a number of moment inequalities. The results are presented here for the sake of completeness however their proofs are omitted since they are based on properties of the $\phi$ functions and are the same as in the single index case. We refer the interested reader to \cite{PR2007} or Section 2.8 in \cite{PR2012} for the details.

\begin{theorem}
	Let $\left\{S_{\mathbf{n}}, \mathbf{n} \in \mathbb{N}^{k}\right\}$ be a nonnegative multiindexed demisubmartingale with $S_{\boldsymbol{\ell}}=0$ when $\prod_{i=1}^{k} \ell_{i}=0$ and let $\phi \in \mathcal{C}$. Then,
	$$
	E\left[\phi\left(\max_{\mathbf{i} \leq \mathbf{n}} S_{\mathbf{i}}\right)\right] \leq \phi(a)+\frac{\lambda}{1-\lambda} E\left[\Phi_{a}\left(\frac{S_{\mathbf{n}}}{\lambda}\right)\right]
	$$
	for all $a \geq 0,0<\lambda<1$ and $\mathbf{n} \in \mathbb{N}^{k}$. 
\end{theorem}

\medskip

\noindent A special case of \eqref{OrlExp} is obtained for $\phi(x) = x$. The proof is obtained by applying the same methodology as in \cite{PR2012} page 64.

\begin{theorem}
	Let $\left\{S_{\mathbf{n}}, \mathbf{n} \in \mathbb{N}^{k}\right\}$ be a nonnegative multiindexed demisubmartingale with $S_{\boldsymbol{\ell}}=0$ when $\prod_{i=1}^{k} \ell_{i}=0$. Then, for any $\mathbf{n} \in \mathbb{N}^{k}$,
	\[
	E\left(\max_{\mathbf{i} \leq \mathbf{n}} S_{\mathbf{i}}\right) \leq b+\frac{b}{b-1}\left(E\left(S_{\mathbf{n}} \log ^{+} S_{n}\right)-E\left(S_{\mathbf{n}}-1\right)^{+}\right), \quad b>1.
	\]
\end{theorem}

\begin{remark}
	Observe that in the case where $b=e$ the inequality above provides a sharper bound compared to the second inequality of Corollary \ref{mom1}.
\end{remark}

\medskip \noindent By utilizing inequality \eqref{Doobv2} and Lemma 2.8.3 in \cite{PR2012} we can easily obtain the following result.

\begin{theorem}
	Let $\left\{S_{\mathbf{n}}, \mathbf{n} \in \mathbb{N}^{k}\right\}$ be a nonnegative multiindexed demisubmartingale with $S_{\boldsymbol{\ell}}=0$ when $\prod_{i=1}^{k} \ell_{i}=0$ and let $\phi \in \mathcal{C}$ with $p_{\phi}=\displaystyle\inf _{x>0} \frac{x \phi^{\prime}(x)}{\phi(x)}>1$. Then for all $\mathbf{n} \in \mathbb{N}^{k}$
	$$
	E\left[\phi\left(\max_{\mathbf{i} \leq \mathbf{n}} S_{\mathbf{i}}\right)\right] \leq E\left[\phi\left(q_{\phi} S_{\mathbf{n}}\right)\right]
	$$
	where $q_{\phi}=\frac{p_{\phi}}{p_{\phi}-1}$.
\end{theorem}

\medskip \noindent The next result is a direct consequence of the previous theorem and the properties of the $\phi$ function.

\begin{theorem}
	Let $\left\{S_{\mathbf{n}}, \mathbf{n} \in \mathbb{N}^{k}\right\}$ be a nonnegative multiindexed demisubmartingale with $S_{\boldsymbol{\ell}}=0$ when $\prod_{i=1}^{k} \ell_{i}=0$. Suppose that the function $\phi \in \mathcal{C}$ is moderate. Then, for any $\mathbf{n} \in \mathbb{N}^{k}$,
	$$
	E\left[\phi\left(\max_{\mathbf{i} \leq \mathbf{n}} S_{\mathbf{i}}\right)\right] \leq E\left[\phi\left(q_{\phi} S_{\mathbf{n}}\right)\right] \leq q_{\Phi}^{p_{\phi}^{*}} E\left[\phi\left(S_{\mathbf{n}}\right)\right].
	$$ 
\end{theorem}

\begin{theorem}
	Let $\left\{S_{\mathbf{n}}, \mathbf{n} \in \mathbb{N}^{k}\right\}$ be a nonnegative multiindexed demisubmartingale with $S_{\boldsymbol{\ell}}=0$ when $\prod_{i=1}^{k} \ell_{i}=0$. Suppose $\phi$ is a nonnegative nondecreasing function on $[0, \infty)$ such that $\phi^{1 / \gamma}$ is also nondecreasing and convex for some $\gamma>1$. Then
	$$
	E\left[\phi\left(\max_{\mathbf{i} \leq \mathbf{n}} S_{\mathbf{i}}\right)\right] \leq\left(\frac{\gamma}{\gamma-1}\right)^{\gamma} E\left[\phi\left(S_{\mathbf{n}}\right)\right].
	$$
	Moreover, for any $r>0$,
	$$E\left[e^{r \max_{\mathbf{i} \leq \mathbf{n}} S_{\mathbf{i}}}\right] \leq e E\left[e^{r S_{\mathbf{n}}}\right].$$
\end{theorem}
\begin{proof}
	For the proof of the first inequality we use the first part of Corollary \ref{mom1} for the sequence $\left\{\left[\phi\left(S_{\mathbf{n}}\right)\right]^{1 / \gamma}, \mathbf{n} \in \mathbb{N}^{k}\right\}$ since this is also a nonnegative multiindexed demisubmartingale. The desired inequality follows by choosing $p = \gamma$. The second result is obtained by setting $\phi(x)  =  e^{rx}$ in the first inequality and by letting $\gamma$ to tend to $\infty$.
\end{proof}

\medskip

\noindent This section is concluded with a maximal inequality for the case where both, the function $\phi$ and its $m$-th derivative for some $m\geq 1$, are Orlicz functions. The proof runs along the same lines as in the case of a single index since the source result is Corollary \ref{mom1}.

\begin{theorem}
	Let $\left\{S_{\mathbf{n}}, \mathbf{n} \in \mathbb{N}^{k}\right\}$ be a nonnegative multiindexed demisubmartingale with $S_{\boldsymbol{\ell}}=0$ when $\prod_{i=1}^{k} \ell_{i}=0$. Let $\phi \in \mathcal{C}$ which is differentiable $m$ times with the $m$-th derivative $\phi^{(m)} \in \mathcal{C}$ for some $m \geq 1$. Then
	$$
	E\left[\phi\left(\max_{\mathbf{i} \leq \mathbf{n}} S_{\mathbf{i}}\right)\right] \leq\left(\frac{m+1}{m}\right)^{m+1} E\left[\phi\left(S_{\mathbf{n}}\right)\right].
	$$
\end{theorem}

\section{Upcrossing Inequality}
We follow the notation of Section 2.4 in \cite{PR2012}. Given a finite set of muliindexed random variables from the set $\left\{S_{\mathbf{n}}, \mathbf{n} \in \mathbb{N}^{k}\right\}$ and $a<b$ we define the sequence of stopping times  in the $s$-th direction as follows
$$
J^{(s)}_{2 m-1}=\left\{\begin{array}{l}
n_s+1 \quad \text { if }\left\{j: J^{(s)}_{2 m-2}<j \leq n_s \text { and } S_{\mathbf{n};s;j} \leq a\right\} \text { is empty } \\
\min \left\{j: J^{(s)}_{2 m-2}<j \leq n_s \text { and } S_{\mathbf{n};s;j} \leq a\right\}, \quad \text { otherwise }
\end{array}\right.
$$
and
$$
J^{(s)}_{2 m}=\left\{\begin{array}{l}
n_s+1 \text { if }\left\{j: J^{(s)}_{2 m-1}<j \leq n \text { and } S_{\mathbf{n};s;j} \geq b\right\} \text { is empty } \\
\min \left\{j: J^{(s)}_{2 m-1}<j \leq n \text { and } S_{\mathbf{n};s;j} \geq b\right\}, \quad \text { otherwise }
\end{array}\right.
$$
for $m=1,2,\ldots$ where $\mathbf{n};s;j = (n_1\, n_2\, \ldots\, n_{s-1} \, j \, n_{s+1}\, \ldots \, n_k)$ and $J^{(s)}_{0} = 0$ for $1\leq s\leq k$. 

\medskip

	\noindent The number of complete upcrossings of the interval $[a, b]$  by the finite sequence of random variables in the $s$-th direction, until the time $n_s,$ is denoted by $U_{n_s}(a,b)$ for $1\leq s\leq k$ where
	$$
	U_{n_s}(a,b)=\max \left\{m: J^{(s)}_{2 m}<n_s+1\right\} .
	$$
	Then, as $n_s\to \infty$, 
	\[
	U_{n_s}(a,b) \to U^{(s)}(a,b) = \max \left\{m: J^{(s)}_{2 m}<\infty\right\} 
	\]
	which denotes the total number of upcrossings in the $s$-th direction. We say that the sequence of random fields has a complete upcrossing if there is a complete upcrossing in at least one direction 
	and the total number of complete upcrossings is defined as 
	\begin{equation}
\label{updef}U(a,b) = \min \{U_{a,b}^{(s)}>0, 1\leq s\leq k\}\quad\mbox{and}\quad U(a,b) = 0\mbox{ if } U_{a,b}^{(s)}=0\mbox{ for all } 1\leq s\leq k.
\end{equation}

\begin{remark}
	It is important to highlight that in the case of random fields where there is no partial order among indices, the concept of upcrossings is not uniquely defined. In the definition considered above, first, we identify the number of upcrossings in each direction and the total number of upcrossings for the process is defined through \eqref{updef}. Consider the example given below for the case $k=2$: suppose that $a<b$ and 
	\[
	x_{11}<a, x_{21} <a, x_{12}>b, x_{22}<b.
	\]
	First,  we consider the case where $s=1$, i.e. consider the upcrossings in the 1-st direction by keeping  the second index constant and allowing the first one to change. Observe that,
	\[
	x_{11}<a\quad\mbox{and}\quad x_{21} <a
	\]
	and
	\[
	x_{12}>b\quad\mbox{and}\quad x_{22}<b
	\]
	i.e. there are no complete upcrossings in the 1-st direction. Next, we study the number of upcrossings in the direction of the second index where the first index remains constant:
	\[
	x_{11}<a\quad\mbox{and}\quad x_{12}>b
	\]
	and
	\[
	x_{21} <a\quad\mbox{and}\quad x_{22}<b.
	\]
	Observe that the first pair gives a complete upcrossing, so in the second direction, there is 1 complete upcrossing. Using the notation introduced above, directionwise we have that 
	\[
	U^{(1)}(a,b) = 0\quad\mbox{and}\quad U^{(2)}(a,b) = 1.
	\]
	\noindent According to \eqref{updef}, the total number of complete upcrossings for the given sequence is $U(a,b) = 1$.	
\end{remark}

The result that follows provides an upper bound for the expected number of complete upcrossings of the interval $[a,b]$ in the $s$-th direction.

\begin{theorem}
	Let $\left\{S_{\mathbf{n}}, \mathbf{n} \in \mathbb{N}^{k}\right\}$ be a multiindexed demisubmartingale with $S_{\boldsymbol{\ell}}=0$ and consider a finite sequence of random variables from this set. Then for $a<b$ and $1\leq s\leq k$,
	$$
	E\left(U_{n_s}(a,b)\right) \leq \frac{E\left(\left(S_{\mathbf{n}}-a\right)^{+}\right)-E\left(\left(S_{\mathbf{n};s;1}-a\right)^{+}\right)}{b-a}
	$$
\end{theorem}
\begin{proof}
	Without loss of generality we assume that $s = k$ and we use the notation $\mathbf{n} = (\mathbf{n}^* \, n_k)$ where $\mathbf{n}^* = (n_1\, \ldots \, n_{k-1})$. The proof is motivated by the proof of Theorem 2.4.1 in Prakasa Rao (2012). For $1 \leq j \leq n_k-1$, define
	$$
	\epsilon_{\mathbf{n}^*,j}= \begin{cases}1 & \text { if for some } m=1,2, \ldots, J^{(s)}_{2 m-2} \leq j<J^{(s)}_{2 m-1} \\ 0 & \text { if for some } m=1,2, \ldots, J^{(s)}_{2 m-1} \leq j<J^{(s)}_{2 m}\end{cases}
	$$
	so that $1-\epsilon_{\mathbf{n}^*,j}$ is the indicator function of the event that the time interval $[j, j+1)$ is a part of an upcrossing possibly incomplete; equivalently
	$$\epsilon_{\mathbf{n}^*,j}= \begin{cases}1 & \text { if either } S_{\mathbf{n}^*,i}>a \text { for } i=1, \ldots, j \text { or } \\ & \text { for some } i=1, \ldots, j, S_{\mathbf{n}^*,i} \geq b \text { and } S_{\mathbf{n}^*k}>a \text { for } k=i+1, \ldots, j \\ =0 & \text { otherwise. }\end{cases}$$
	
	\noindent Let $\Lambda$ be the event that the sequence ends with an incomplete upcrossing, that is, $\tilde{J}^{(s)} \equiv J_{2 U^{(s)}_{a, b}+1}<n_k$. Note that
	$$
	\left(S_{\mathbf{n}}-a\right)^{+}-\left(S_{\mathbf{n}^*1}-a\right)^{+}=\sum_{j=1}^{n_k-1}\left[\left(S_{\mathbf{n}^*,j+1}-a\right)^{+}-\left(S_{\mathbf{n}^*,j}-a\right)^{+}\right]=M_{u}+M_{d}
	$$
	where
	$$
	M_{d}=\sum_{j=1}^{n_k-1} \epsilon_{\mathbf{n}^*,j}\left[\left(S_{\mathbf{n}^*,j+1}-a\right)^{+}-\left(S_{\mathbf{n}^*,j}-a\right)^{+}\right] \geq \sum_{j=1}^{n_k-1} \epsilon_{\mathbf{n}^*,j}\left(S_{\mathbf{n}^*,j+1}-S_{\mathbf{n}^*,j}\right)
	$$
	where the inequality follows from the observation that
	$$
	\left(S_{\mathbf{n}^*,j+1}-a\right)^{+} \geq S_{\mathbf{n}^*,j+1}-a
	$$
	and because of the definition of $\epsilon_{\mathbf{n}^*,j}$ we also have that
	$$
	\epsilon_{\mathbf{n}^*,j}\left(S_{\mathbf{n}^*,j}-a\right)^{+}=\epsilon_{\mathbf{n}^*,j}\left(S_{\mathbf{n}^*,j}-a\right).
	$$
	Observe that
	$$
	\begin{aligned}
	M_{u} &=\sum_{j=1}^{n_k-1}\left(1-\epsilon_{\mathbf{n}^*,j}\right)\left[\left(S_{\mathbf{n}^*,j+1}-a\right)^{+}-\left(S_{\mathbf{n}^*,j}-a\right)^{+}\right] \\
	&=\sum_{k=1}^{U_{n_s}(a,b)} \sum_{j=J^{(s)}_{2 k-1}}^{J^{(s)}_{2 k}-1}\left[\left(S_{\mathbf{n}^*,j+1}-a\right)^{+}-\left(S_{\mathbf{n}^*,j}-a\right)^{+}\right]+\sum_{j=\tilde{J}^{(s)}}^{n-1}\left[\left(S_{\mathbf{n}^*,j+1}-a\right)^{+}-\left(S_{\mathbf{n}^*,j}-a\right)^{+}\right] \\
	&=\sum_{k=1}^{U_{n_s}(a,b)}\left[\left(S_{\mathbf{n}^*,J^{(s)}_{2 k}}-a\right)^{+}-\left(S_{\mathbf{n}^*,J^{(s)}_{2 k-1}}-a\right)^{+}\right]+\left[\left(S_{\mathbf{n}}-a\right)^{+}-\left(S_{\mathbf{n}^*,\tilde{J}^{(s)}}-a\right)^{+}\right] I_{\Lambda} \\
	&=\sum_{k=1}^{U_{n_s}(a,b)}\left(S_{\mathbf{n}^*,J^{(s)}_{2 k}}-a\right)^{+}+\left(S_{\mathbf{n}}-a\right)^{+} I_{\Lambda} \\
	\geq &(b-a) U_{n_s}(a,b).
	\end{aligned}
	$$
	By taking expectations, we have that
	$$
	E\left[\left(S_{\mathbf{n}}-a\right)^{+}-\left(S_{\mathbf{n}^*,1}-a\right)\right] \geq(b-a) EU_{n_s}(a,b)+\sum_{j=1}^{n_k-1} E\left[\epsilon_{\mathbf{n}^*,j}\left(S_{\mathbf{n}^*,j+1}-S_{\mathbf{n}^*,j}\right)\right].
	$$
	The desired result follows by noting that $\epsilon_{\mathbf{n}^*,j}$ is a nonnegative nondecreasing function of $\{S_{\mathbf{n}^*,i},\, i=1,\ldots,j\}$ which forms a multiindexed demisubmartingale and thus the latter term is always nonnegative. 
\end{proof}

\medskip

\noindent The upcrossing inequality leads to the following convergence result.

\begin{theorem}
	Let $\left\{S_{\mathbf{n}}, \mathbf{n} \in \mathbb{N}^{k}\right\}$ be a multiindexed demimartingale with $S_{\boldsymbol{\ell}}=0$ such that
	\begin{equation}
	\label{ass1}\limsup _{\mathbf{n} \rightarrow \infty} E\left|S_{\mathbf{n}}\right|<\infty,
	\end{equation}
	then the $S_{\mathbf{n}}$ converge a.s. and $E|\displaystyle\lim _{\mathbf{n} \rightarrow \infty}S_{\mathbf{n}}|<\infty$.
\end{theorem}
\begin{proof}
It is known that a sequence of real numbers converges to $\mathbb{R}\cup{\pm \infty}$ if and only if the number of upcrossings of $[a,b]$ is finite for all rational numbers $a<b$. By the upcrossing inequality proven above, we have that for every $1\leq s\leq k$
		\[
		EU_{n_s}(a,b)\leq \frac{1}{b-a}E\left(S_{\mathbf{n}}-a\right)^{+} \leq \frac{1}{b-a}(E|S_\mathbf{n}|+a)
		\]
		and  therefore
		\[
		EU^{(s)}_{a, b}\leq \frac{1}{b-a}(\sup_{\mathbf{n}}E|S_\mathbf{n}|+a)
		\]
		and thus $U^{(s)}_{a, b}$ is finite almost surely for every $1\leq s\leq k$ due to condition \eqref{ass1}. By the definition of $U(a,b)$ we have that $U(a,b)\leq U^{(s)}_{a, b}$ for any $s$ and thus $U(a,b)$ is also finite almost surely. Define $\Omega_0 = \displaystyle\bigcap_{a<b\in \mathbb{Q}}\left(U(a, b)<\infty\right)$. Then $P(\Omega_0) = 1$. This observation establishes the convergence of $S_{\mathbf{n}}$ while the last argument for the expectation follows by applying Fatou's lemma and utilizing assumption \eqref{ass1}.
\end{proof}

\section{Whittle-type inequality}
 The next result is the multiindexed analogue of the Whittle type maximal inequality for demimartingales proven by Prakasa Rao in \cite{PR2002} (see also Theorem 2.6.1 in \cite{PR2012}).

\begin{theorem}
	\label{WhittleNond}Let $\left\{S_{\mathbf{n}}, \mathbf{n} \in \mathbb{N}^{k}\right\}$ be a multiindexed demimartingale with $S_{\boldsymbol{\ell}}=0$ when $\prod_{i=1}^{k} \ell_{i}=0$. Let $\phi(\cdot)$ be a nonnegative nondecreasing convex function such that $\phi(0)= 0$. Let $\psi(u)$ be a positive nondecreasing function for $u>0$. Then,
	\begin{equation}
	\label{Whittle} P(\phi(S_{\mathbf{j}})\leq \psi(u_{\mathbf{j}}), \, \mathbf{j}\leq \mathbf{n}) \geq 1-\min_{1\leq s\leq k}\sum_{i =1}^{n_s} \frac{E\left(  \phi\left(S_{\mathbf{n};s;i}\right)- \phi\left(S_{\mathbf{n};s;i-1}\right)\right)}{\psi \left( u_{\mathbf{n};s;i}     \right)}
	\end{equation}
	for $0<u_{\mathbf{i}}\leq u_{\mathbf{j}},\, \mathbf{i}\leq \mathbf{j}$. 
\end{theorem}
\begin{proof}
	The proof will be given for $k=2$ for simplicity. Define
	\[
	A_{n_1n_2} = \{\phi(S_{ij})\leq \psi(u_{ij}), \, (i,j)\leq (n_1,n_2)\}.
	\]
	Then,
	\begin{align*}
	&P(A_{n_1n_2}) = E\left( \prod_{\ell =1}^{n_1} I(A_{\ell n_2})\right)  = E\left( \prod_{\ell =1}^{n_1-1} I(A_{\ell n_2})I(A_{n_1n_2})\right)\geq E\left( \prod_{\ell =1}^{n_1-1} I(A_{\ell n_2})\left( 1-\frac{\phi(S_{n_1n_2})}{\psi(u_{n_1n_2})} \right)               \right)
	\end{align*}
	where the inequality follows by observing that 
	\[
	I(A_{n_1n_2}) \geq 1-\frac{\phi(S_{n_1n_2})}{\psi(u_{n_1n_2})}.
	\]
	Notice that 
	\begin{align*}
	&E\left(    \prod_{\ell =1}^{n_1-1} I(A_{\ell n_2}) \left[ \left(  1- \frac{\phi(S_{n_1n_2})}{\psi(u_{n_1n_2})}\right)         -   \left(  1- \frac{\phi(S_{n_1-1n_2})}{\psi(u_{n_1n_2})}\right)                   \right]   + \frac{\phi(S_{n_1n_2}) - \phi(S_{n_1-1n_2})}{\psi(u_{n_1n_2})}    \right)\\
	&= E\left(  \left(   1- \prod_{\ell =1}^{n_1-1} I(A_{\ell n_2})      \right) \frac{\phi(S_{n_1n_2}) - \phi(S_{n_1-1n_2})}{\psi(u_{n_1n_2})}      \right)\geq 0
	\end{align*}
	since $\left(   1- \prod_{\ell =1}^{n_1-1} I(A_{\ell n_2})      \right)$ is nonnegative componentwise nondecreasing function of $\{ \phi(S_{in_2}), \, 1\leq i\leq n_1\}$ which forms a single-index demimartingale. Hence,
	\begin{align*}
	&P(A_{n_1n_2}) \geq E\left(    \prod_{\ell =1}^{n_1-1} I(A_{\ell n_2}) \left(     1- \frac{\phi(S_{n_1-1n_2})}{\psi(u_{n_1n_2})}        \right)           \right) - \frac{E\phi(S_{n_1n_2}) - E\phi(S_{n_1-1n_2})}{\psi(u_{n_1n_2})}\\
	&\geq E\left(    \prod_{\ell =1}^{n_1-2} I(A_{\ell n_2}) \left(     1- \frac{\phi(S_{n_1-1n_2})}{\psi(u_{n_1-1n_2})}        \right)           \right) - \frac{E\phi(S_{n_1n_2}) - E\phi(S_{n_1-1n_2})}{\psi(u_{n_1n_2})}
	\end{align*}
	where the last inequality follows due to the fact that $u_{\mathbf{i}}\leq u_{\mathbf{j}}$ for $\mathbf{i}\leq \mathbf{j}$ and $\psi(u_\mathbf{n})$ is a nondecreasing sequence of positive numbers. Continuing with the same manner we have that 
	\[
	P(A_{n_1n_2}) \geq 1 - \sum_{\ell = 1}^{n_1} \frac{E\phi(S_{\ell n_2}) - E\phi(S_{\ell-1n_2})}{\psi(u_{\ell n_2})}.
	\]
	Similarly we can obtain
	\[
	P(A_{n_1n_2}) \geq 1 - \sum_{\ell = 1}^{n_2} \frac{E\phi(S_{n_1\ell}) - E\phi(S_{n_1\ell-1})}{\psi(u_{n_1\ell })}
	\]
	and these last two inequalities lead to the 2-index analogue of \eqref{Whittle}.
\end{proof}

\medskip

\noindent The next result provides a generalization of the Whittle type inequality where the assumption of nondecreasing property is dropped. 

\begin{theorem}
	\label{GenWhittle}Let $\left\{S_{\mathbf{n}}, \mathbf{n} \in \mathbb{N}^{k}\right\}$ be a multiindexed demimartingale with $S_{\boldsymbol{\ell}}=0$ when $\prod_{i=1}^{k} \ell_{i}=0$. Let $\phi(\cdot)$ be a nonnegative convex function such that $\phi(0)= 0$. Let $\psi(u)$ be a positive nondecreasing function for $u>0$ and 
	\[
	A_{\mathbf{n}} = \{\phi(S_{\mathbf{j}})\leq \psi(u_{\mathbf{j}}), \, \mathbf{j}\leq \mathbf{n}\}
	\]
	where $0<u_{\mathbf{i}}\leq u_{\mathbf{j}}$ for $\mathbf{i}\leq \mathbf{j}$. Then,
	\begin{equation}
	\label{Whittle2} P(A_{\mathbf{n}}) \geq 1-\min_{1\leq s\leq k}\sum_{i =1}^{n_s} \frac{E\left(  \phi\left(S_{\mathbf{n};s;i}\right)- \phi\left(S_{\mathbf{n};s;i-1}\right)\right)}{\psi \left( u_{\mathbf{n};s;i}     \right)}.
	\end{equation}
\end{theorem}
\begin{proof}
	The proof will be given for $k=2$. Similar to the proof of Theorem \ref{Chownew}, we start by defining the functions 
	\[
	u(x) = \phi(x)I\{x\geq 0\}\mbox{ and }v(x)=\phi(x)I\{x<0\}.
	\]
	Recall that both $u(x)$ and $v(x)$ are nonnegative convex functions, however they have different monotonicity since $u(x)$ is nondecreasing while $v(x)$ is  nonincreasing. Moreover,
	\[
	\phi(x) = u(x)+v(x) = \max\{u(x),v(x)\}.
	\]
	Then, we have that
	\begin{eqnarray*}
		P\left(\max_{(i,j)\leq (n_1,n_2)}\frac{\phi(S_{ij})}{\psi(u_{ij})}> 1\right)&=&P\left(\max_{(i,j)\leq (n_1,n_2)}\frac{\max\{u(S_{ij}),v(S_{ij})\}}{\psi(u_{ij})}> 1\right)\\
		&\leq&P\left(\max_{(i,j)\leq (n_1,n_2)}\frac{u(S_{ij})}{\psi(u_{ij})}> 1\right)+P\left(\max_{(i,j)\leq (n_1,n_2)}\frac{v(S_{ij})}{\psi(u_{ij})}> 1\right).
	\end{eqnarray*}
	The above inequality can be written as
	\[
	P\left(\max_{(i,j)\leq (n_1,n_2)}\frac{\phi(S_{ij})}{\psi(u_{ij})}\leq 1\right)\geq P\left(\max_{(i,j)\leq (n_1,n_2)}\frac{u(S_{ij})}{\psi(u_{ij})}\leq 1\right)+P\left(\max_{(i,j)\leq (n_1,n_2)}\frac{v(S_{ij})}{\psi(u_{ij})}\leq 1\right)-1
	\]
	or equivalently,
	\begin{equation}
	\label{guv}P\left(\frac{\phi(S_{ij})}{\psi(u_{ij})}\leq 1,\, (i,j)\leq (n_1,n_2)\right)\geq P\left(\frac{u(S_{ij})}{\psi(u_{ij})}\leq 1,\, (i,j)\leq (n_1,n_2) \right)+P\left(\frac{v(S_{ij})}{\psi(u_{ij})}\leq 1,\, (i,j)\leq (n_1,n_2)\right)-1.
	\end{equation}
	Since $u$ is a nondecreasing convex function by Theorem \ref{WhittleNond} we have that
	\begin{equation}
	\label{uRel1}P\left(\frac{u(S_{ij})}{\psi(u_{ij})}\leq 1,\, (i,j)\leq (n_1,n_2) \right)\geq 1 - \sum_{\ell = 1}^{n_1} \frac{E(u(S_{\ell n_2}) - u(S_{\ell-1n_2}))}{\psi(u_{\ell n_2})}.
	\end{equation}
	Similarly we can obtain
	\begin{equation}
	\label{uRel2}
	P\left(\frac{u(S_{ij})}{\psi(u_{ij})}\leq 1,\, (i,j)\leq (n_1,n_2) \right) \geq 1 - \sum_{\ell = 1}^{n_2} \frac{E(u(S_{n_1\ell}) - u(S_{n_1\ell-1}))}{\psi(u_{n_1\ell })}
	\end{equation}	
	Let $B_{n_1n_2} = \left\{\frac {v(S_{ij})}{\psi(u_{ij})}\leq 1,\, (i,j)\leq (n_1,n_2)\right\}$. Following similar arguments as in the proof of Theorem \ref{WhittleNond} we have that 
	\begin{align*}
	&E\left(    \prod_{\ell =1}^{n_1-1} I(B_{\ell n_2}) \left[ \left(  1- \frac{v(S_{n_1n_2})}{\psi(u_{n_1n_2})}\right)         -   \left(  1- \frac{v(S_{n_1-1n_2})}{\psi(u_{n_1n_2})}\right)                   \right]   + \frac{v(S_{n_1n_2}) - v(S_{n_1-1n_2})}{\psi(u_{n_1n_2})}    \right)\\
	&= E\left(  \left(   1- \prod_{\ell =1}^{n_1-1} I(B_{\ell n_2})      \right) \frac{v(S_{n_1n_2}) - v(S_{n_1-1n_2})}{\psi(u_{n_1n_2})}      \right)\geq E\left(  \left(   1- \prod_{\ell =1}^{n_1-1} I(B_{\ell n_2})      \right) \frac{(S_{n_1n_2} - S_{n_1-1n_2})}{\psi(u_{n_1n_2})} f(S_{n_1-1n_2})     \right)\\
	&\geq 0
	\end{align*}
	where the first inequality is due to the convexity of the function $v$ with $f(.)$ being the left derivative of $v$ which is a nonpositive nondecreasing function. The last inequality is obtained by the multiindexed demimartingale property since 
	\[
	\left(   1- \prod_{\ell =1}^{n_1-1} I(B_{\ell n_2}) \right)f(S_{n_1-1n_2}) 
	\]
	is a nondecreasing function of $\{S_{in_2}, \,i=1,\ldots,n_1-1\}$. Working in a similar manner as in Theorem \ref{WhittleNond} we can obtain the following bounds for the $P(B_{n_1n_2})$,
	\begin{equation}
	\label{vRel1}P\left(\frac{v(S_{ij})}{\psi(u_{ij})}\leq 1,\, (i,j)\leq (n_1,n_2) \right)\geq 1 - \sum_{\ell = 1}^{n_1} \frac{E(v(S_{\ell n_2}) - v(S_{\ell-1n_2}))}{\psi(u_{\ell n_2})}
	\end{equation}
	and
	\begin{equation}
	\label{vRel2}
	P\left(\frac{v(S_{ij})}{\psi(u_{ij})}\leq 1,\, (i,j)\leq (n_1,n_2) \right) \geq 1 - \sum_{\ell = 1}^{n_2} \frac{E(v(S_{n_1\ell}) - v(S_{n_1\ell-1}))}{\psi(u_{n_1\ell })}
	\end{equation}	
	By combining \eqref{guv}, \eqref{uRel1} and \eqref{vRel1} we have
	\[
	P\left(\frac{\phi(S_{ij})}{\psi(u_{ij})}\leq 1,\, (i,j)\leq (n_1,n_2)\right)\geq 1 - \sum_{\ell = 1}^{n_1} \frac{E(\phi(S_{\ell n_2}) - \phi(S_{\ell-1n_2}))}{\psi(u_{\ell n_2})}
	\]
	while by combining \eqref{guv}, \eqref{uRel2} and \eqref{vRel2} we get
	\[
	P\left(\frac{\phi(S_{ij})}{\psi(u_{ij})}\leq 1,\, (i,j)\leq (n_1,n_2)\right)\geq 1 - \sum_{\ell = 1}^{n_2} \frac{E(\phi(S_{n_1\ell}) - \phi(S_{n_1\ell-1}))}{\psi(u_{n_1\ell })}.
	\]
	The desired result follows by combining the last two inequalities.	
\end{proof}

\medskip

\noindent As a direct consequence of the above Whittle-type inequality we have the corollary that follows.

\begin{corollary}
	Let $\left\{S_{\mathbf{n}}, \mathbf{n} \in \mathbb{N}^{k}\right\}$ be a multiindexed demimartingale with $S_{\boldsymbol{\ell}}=0$ when $\prod_{i=1}^{k} \ell_{i}=0$. Let $\phi(\cdot)$ be a nonnegative convex function such that $\phi(0)= 0$. Let $\psi(u)$ be a positive nondecreasing function for $u>0$.  Then, for $\epsilon>0$,
	\[
	\epsilon P\left(  \sup_{\mathbf{j}\leq \mathbf{n}}\frac{\phi(S_\mathbf{j})}{\psi(u_{\mathbf{j}})}\geq \epsilon          \right) \leq \min_{1\leq s\leq k}\sum_{i =1}^{n_s} \frac{E\left(  \phi\left(S_{\mathbf{n};s;i}\right)- \phi\left(S_{\mathbf{n};s;i-1}\right)\right)}{\psi \left( u_{\mathbf{n};s;i}     \right)}.
	\]
\end{corollary}
\begin{proof}
	The result follows by first writing that
	\[
	P\left(  \sup_{\mathbf{j}\leq \mathbf{n}}\frac{\phi(S_\mathbf{j})}{\psi(u_{\mathbf{j}})}\geq \epsilon          \right) = 1- P\left( \frac{\phi(S_\mathbf{j})}{\psi(u_{\mathbf{j}})}\leq \epsilon , \forall \,  \mathbf{j}\leq \mathbf{n}        \right)
	\]
	and applying the result of the previous theorem. 
\end{proof}

\medskip

\noindent The Whittle-type inequality can also be employed to obtain the following convergence result. 

\begin{theorem}
	$\left\{S_{\mathbf{n}}, \mathbf{n} \in \mathbb{N}^{k}\right\}$ be a multiindexed demimartingale with $S_{\boldsymbol{\ell}}=0$ when $\prod_{i=1}^{k} \ell_{i}=0$. Let $\phi(\cdot)$ be a nonnegative convex function such that $\phi(0)= 0$. Let $\psi(u)$ be a positive nondecreasing function for $u>0$ such that $\psi(u) \to \infty$ as $u\to \infty$. Further, suppose that there is $s \in \{1,2,\ldots, k \}$ such that 
	\begin{equation}
	\label{cond1} \displaystyle\sum_{i =1}^{\infty} \frac{E\left(  \phi\left(S_{\mathbf{n};s;i}\right)- \phi\left(S_{\mathbf{n};s;i-1}\right)\right)}{\psi \left( u_{\mathbf{n};s;i}     \right)}<\infty
	\end{equation}
	for any nondecreasing sequence $u_{\mathbf{n}} \to \infty$ as $\mathbf{n} \to \infty$. Then,
	\[
	\frac{\phi(S_{\mathbf{n}})}{\psi(u_{\mathbf{n}})} \to 0 \quad \mbox{a.s. for}\quad \mathbf{n} \rightarrow \infty.
	\]
\end{theorem}
\begin{proof}
	For simplicity, we assume that $k=2$ and without loss of generality we assume that \eqref{cond1} is satisfied for $s=2$. Then,
	\begin{align*}
	& P\left(  \sup_{(i,j) \geq (n_1, n_2)}  \frac{\phi(S_{ij})}{\psi(u_{ij} )}\geq \epsilon        \right) \leq   \sum_{\ell = 1}^{\infty} \frac{E\phi(S_{n_1\ell}) - E\phi(S_{n_1\ell-1})}{\epsilon\psi(u_{n_1\ell })}
	\leq \frac{E(\phi(S_{n_1n_2}))}{\epsilon\psi(u_{n_1n_2})} +   \sum_{\ell = n_2+1}^{\infty} \frac{E\phi(S_{n_1\ell}) - E\phi(S_{n_1\ell-1})}{\epsilon\psi(u_{n_1\ell })} .
	\end{align*}
	Notice that due to the given assumptions both summands converge to zero as $ \mathbf{n} \rightarrow \infty$ which leads to the desired convergence.
\end{proof}

\medskip

\noindent{\textbf{Acknowledgement}} Work of the second author was supported  under the ``INSA Senior Scientist" scheme at the CR Rao Advanced Institute of Mathematics, Statistics and Computer Science, Hyderabad, India.

\end{document}